\documentclass{amsart}
\usepackage{amsmath}
\usepackage{amsfonts}

\newtheorem{theorem}{Theorem}
\theoremstyle{plain}

\newtheorem{corollary}{Corollary}

\newtheorem{definition}{Definition}

\newtheorem{lemma}{Lemma}

\newtheorem{proposition}{Proposition}
\newtheorem{remark}{Remark}

\numberwithin{equation}{section}
\input{tcilatex}

\begin{document}
\title[On $s$-GA-Convex Functions]{Hermite-Hadamard type inequalities for GA-%
$s$-convex functions }
\author{\.{I}mdat \.{I}\c{s}can}
\address{Department of Mathematics, Faculty of Arts and Sciences,\\
Giresun University, 28100, Giresun, Turkey.}
\email{imdati@yahoo.com, imdat.iscan@giresun.edu.tr}

\begin{abstract}
In this paper, The author introduces the concepts of the GA-$s$-convex
functions in the first sense and second sense and establishes some integral
inequalities of Hermite-Hadamard type related to the GA-$s$-convex functions.
\end{abstract}

\maketitle

\section{Introduction}

In this section, we firstly list several definitions and some known results.

The following concept was introduced by Orlicz in \cite{O61}:

\begin{definition}
Let $0<s\leq 1$. A function $f:I\subseteq 
\mathbb{R}
_{+}\rightarrow 
\mathbb{R}
$ where $%
\mathbb{R}
_{+}=\left[ 0,\infty \right) $, is said to be $s$-convex in the first sense
if%
\begin{equation*}
f\left( \alpha x+\beta y\right) \leq \alpha ^{s}f(x)+\beta ^{s}f(y)
\end{equation*}%
for all $x,y\in I$ and $\alpha ,\beta \geq 0$ with $\alpha ^{s}+\beta ^{s}=1$%
. We denote this class of real functions by $K_{s}^{1}.$
\end{definition}

In \cite{HM94}, Hudzik and Maligranda considered the following class of
functions:

\begin{definition}
A function $f:I\subseteq 
\mathbb{R}
_{+}\rightarrow 
\mathbb{R}
$ where $%
\mathbb{R}
_{+}=\left[ 0,\infty \right) $, is said to be $s$-convex in the second sense
if%
\begin{equation*}
f\left( \alpha x+\beta y\right) \leq \alpha ^{s}f(x)+\beta ^{s}f(y)
\end{equation*}%
for all $x,y\in I$ and $\alpha ,\beta \geq 0$ with $\alpha +\beta =1$ and $s$
fixed in $\left( 0,1\right] $. They denoted this by $K_{s}^{2}.$
\end{definition}

It can be easily seen that for $s=1$, $s$-convexity reduces to ordinary
convexity of functions defined on $[0,\infty )$.

In \cite{DF99}, Dragomir and Fitzpatrick proved a variant of
Hermite-Hadamard inequality which holds for the $s$-convex functions.

\begin{theorem}
Suppose that $f:%
\mathbb{R}
_{+}\mathbb{\rightarrow }%
\mathbb{R}
_{+}$ is an $s$-convex function in the second sense, where $s\in \lbrack
0,1) $ and let $a,b\in \lbrack 0,\infty )$, $a<b$. If $f\in L\left[ a,b%
\right] $, then the following inequalities hold 
\begin{equation}
2^{s-1}f\left( \frac{a+b}{2}\right) \leq \frac{1}{b-a}\dint%
\limits_{a}^{b}f(x)dx\leq \frac{f(a)+f(b)}{s+1}\text{.}  \label{1-1}
\end{equation}%
the constant $k=\frac{1}{s+1}$ is the best possible in the second inequality
in (\ref{1-1}).
\end{theorem}

The above inequalities are sharp. For recent results and generalizations
concerning $s$-convex functions see \cite{ADK11,DF99,HBI09,I13b,KBO07}

\begin{definition}[\protect\cite{N00,N03}]
A function $f:I\subseteq \mathbb{R}_{+}\mathbb{\rightarrow R}$ is said to be
a GA-convex function on $I$ if%
\begin{equation*}
f(x^{t}y^{1-t})\leq tf(x)+(1-t)f(y)
\end{equation*}%
holds for all $x,y\in I$ and $t\in \left[ 0,1\right] $, where $x^{t}y^{1-t}$
and $tf(x)+(1-t)f(y)$ are respectively called the weighted geometric mean of
two positive numbers $x$ and $y$ and the weighted arithmetic mean of $f(x)$
and $f(y)$.
\end{definition}

For $b>a>0$, let $G\left( a,b\right) =\sqrt{ab}$, $L\left( a,b\right)
=\left( b-a\right) /\left( \ln b-\ln a\right) $, $I\left( a,b\right) =\left(
1/e\right) \left( b^{b}/a^{a}\right) ^{1/(b-a)}$, $A\left( a,b\right) =\frac{%
a+b}{2}$, and $L_{p}\left( a,b\right) =\left( \frac{b^{p+1}-a^{p+1}}{%
(p+1)(b-a)}\right) ^{\frac{1}{p}}$,$\ p\in 
\mathbb{R}
\backslash \left\{ -1,0\right\} $, be the geometric, logarithmic, identric,
arithmetic and p-logarithmic means of $a$ and $b$, respectively. Then 
\begin{equation*}
\min \left\{ a,b\right\} <G\left( a,b\right) <L\left( a,b\right) <I\left(
a,b\right) <A\left( a,b\right) <\max \left\{ a,b\right\} .
\end{equation*}

In \cite{ZJQ13}, Zhang et al. established some Hermite-Hadamard type
integral inequalities for GA-convex functions and applied these inequalities
to construct several inequalities for special means and they used the
following lemma to prove their results:

\begin{lemma}
Let $f:I\subseteq \mathbb{R}_{+}\mathbb{\rightarrow R}$ be a differentiable
mapping on $I^{\circ }$, and $a,b\in I^{\circ }$,with $a<b$. If $f^{\prime
}\in L\left[ a,b\right] $, then%
\begin{equation*}
bf(b)-af(a)-\dint\limits_{a}^{b}f(x)dx=\left( \ln b-\ln a\right)
\dint\limits_{0}^{1}b^{2t}a^{2\left( 1-t\right) }f^{\prime }\left(
b^{t}a^{1-t}\right) dt.
\end{equation*}
\end{lemma}

Also, the main inequalities in \cite{ZJQ13} are pointed out as follows:

\begin{theorem}
Let $f:I\subseteq \mathbb{R}_{+}\mathbb{\rightarrow R}$ be differentiable on 
$I^{\circ }$, and $a,b\in I$ with $a<b$ and $f^{\prime }\in L\left[ a,b%
\right] .$ If $\left\vert f^{\prime }\right\vert ^{q}$ is GA-convex on $%
\left[ a,b\right] $ for $q\geq 1$, then%
\begin{equation*}
\left\vert bf(b)-af(a)-\dint\limits_{a}^{b}f(x)dx\right\vert \leq \frac{%
\left[ \left( b-a\right) A\left( a,b\right) \right] ^{1-1/q}}{2^{1/q}}
\end{equation*}%
\begin{equation*}
\times \left\{ \left[ L(a^{2},b^{2})-a^{2}\right] \left\vert f^{\prime
}(a)\right\vert ^{q}+\left[ b^{2}-L(a^{2},b^{2})\right] \left\vert f^{\prime
}(b)\right\vert ^{q}\right\} ^{1/q}.
\end{equation*}
\end{theorem}

\begin{theorem}
Let $f:I\subseteq \mathbb{R}_{+}\mathbb{\rightarrow R}$ be differentiable on 
$I^{\circ }$, and $a,b\in I$ with $a<b$ and $f^{\prime }\in L\left[ a,b%
\right] .$ If $\left\vert f^{\prime }\right\vert ^{q}$ is GA-convex on $%
\left[ a,b\right] $ for $q>1$, then%
\begin{equation*}
\left\vert bf(b)-af(a)-\dint\limits_{a}^{b}f(x)dx\right\vert \leq \left( \ln
b-\ln a\right)
\end{equation*}%
\begin{equation*}
\times \left[ L(a^{2q/(q-1)},b^{2q/(q-1)})-a^{2q/(q-1)}\right] ^{1-1/q}\left[
A\left( \left\vert f^{\prime }(a)\right\vert ^{q},\left\vert f^{\prime
}(b)\right\vert ^{q}\right) \right] ^{1/q}.
\end{equation*}
\end{theorem}

\begin{theorem}
Let $f:I\subseteq \mathbb{R}_{+}\mathbb{\rightarrow R}$ be differentiable on 
$I^{\circ }$, and $a,b\in I$ with $a<b$ and $f^{\prime }\in L\left[ a,b%
\right] .$ If $\left\vert f^{\prime }\right\vert ^{q}$ is GA-convex on $%
\left[ a,b\right] $ for $q>1$ and $2q>p>0$, then%
\begin{equation*}
\left\vert bf(b)-af(a)-\dint\limits_{a}^{b}f(x)dx\right\vert \leq \frac{%
\left( \ln b-\ln a\right) ^{1-1/q}}{p^{1/q}}
\end{equation*}%
\begin{eqnarray*}
&&\times \left[ L(a^{(2q-p)/(q-1)},b^{(2q-p)/(q-1)})\right] ^{1-1/q} \\
&&\times \left\{ \left[ L(a^{p},b^{p})-a^{p}\right] \left\vert f^{\prime
}(a)\right\vert ^{q}+\left[ b^{p}-L(a^{p},b^{p})\right] \left\vert f^{\prime
}(b)\right\vert ^{q}\right\} ^{1/q}.
\end{eqnarray*}
\end{theorem}

In \cite{ZCZ10}, Zhang et al. established the following Hermite-Hadamard
type inequality for GA-convex (concave) functions:

\begin{theorem}
If $b>a>0$ and $f:\left[ a,b\right] \rightarrow 
\mathbb{R}
$ is a differentiable GA-convex (concave) function then%
\begin{equation*}
f\left( I(a,b)\right) \leq (\geq )\frac{1}{b-a}\dint\limits_{a}^{b}f(x)dx%
\leq (\geq )\frac{b-L(a,b)}{b-a}f(b)+\frac{L(a,b)-a}{b-a}f(a).
\end{equation*}
\end{theorem}

In \cite{I13}, the author proved the following identity and established some
new Hermite-Hadamard-like type inequalities for the geometrically convex
functions.

\begin{lemma}
\label{1.1}Let $f:I\subseteq \left( 0,\infty \right) \mathbb{\rightarrow }%
\mathbb{R}
$ be a differentiable mapping on $I^{\circ }$, and $a,b\in I$,with $a<b$. If 
$f^{\prime }\in L\left[ a,b\right] $, then%
\begin{equation*}
f\left( \sqrt{ab}\right) -\frac{1}{\ln b-\ln a}\dint\limits_{a}^{b}\frac{f(x)%
}{x}dx
\end{equation*}%
\begin{equation*}
=\frac{\left( \ln b-\ln a\right) }{4}\left[ a\dint\limits_{0}^{1}t\left( 
\frac{b}{a}\right) ^{\frac{t}{2}}f^{\prime }\left( a^{1-t}\left( ab\right) ^{%
\frac{t}{2}}\right) dt-b\dint\limits_{0}^{1}t\left( \frac{a}{b}\right) ^{%
\frac{t}{2}}f^{\prime }\left( b^{1-t}\left( ab\right) ^{\frac{t}{2}}\right)
dt\right] ,
\end{equation*}%
\begin{equation*}
\frac{f(a)+f(b)}{2}-\frac{1}{\ln b-\ln a}\dint\limits_{a}^{b}\frac{f(x)}{x}dx
\end{equation*}%
\begin{equation*}
=\frac{\left( \ln b-\ln a\right) }{2}\left[ a\dint\limits_{0}^{1}t\left( 
\frac{b}{a}\right) ^{t}f^{\prime }\left( a^{1-t}b^{t}\right)
dt-b\dint\limits_{0}^{1}t\left( \frac{a}{b}\right) ^{t}f^{\prime }\left(
b^{1-t}a^{t}\right) dt\right]
\end{equation*}
\end{lemma}

In this paper, we will give concepts $s$-GA-convex functions in the first
and second sense and establish some new integral inequalities of
Hermite-Hadamard-like type for these classes of functions by using Lemma \ref%
{1.1}.

\section{Definitions of GA-$s$-convex functions in the first and second sense%
}

Now it is time to introduce two concepts, GA-$s$-convex functions in the
first and second sense.

\begin{definition}
Let $0<s\leq 1$. A function $f:I\subseteq \mathbb{R}_{+}\mathbb{\rightarrow R%
}$ is said to be a GA-$s$-convex (concave) function in the first sense on $I$
if%
\begin{equation*}
f(x^{t}y^{1-t})\leq (\geq )\ t^{s}f(x)+(1-t^{s})f(y)
\end{equation*}%
holds for all $x,y\in I$ and $t\in \left[ 0,1\right] $.
\end{definition}

\begin{definition}[{\protect\cite[Definition 2.1]{SYQ13}}]
Let $0<s\leq 1$. A function $f:I\subseteq \mathbb{R}_{+}\mathbb{\rightarrow R%
}$ is said to be a GA-$s$-convex (concave) function in the second sense on $I
$ if%
\begin{equation*}
f(x^{t}y^{1-t})\leq (\geq )\ t^{s}f(x)+(1-t)^{s}f(y)
\end{equation*}%
holds for all $x,y\in I$ and $t\in \left[ 0,1\right] $.
\end{definition}

It is clear that when $s=1$, GA-$s$-convex functions in the first and second
sense become GA-convex functions.

\section{Inequalities for GA-$s$-convex functions in the first and second
sense}

Now we are in a position to establish some inequalities of Hermite--Hadamard
type for GA-$s$-convex functions in the first and second sense

\begin{theorem}
\label{3.1} Let $0<s\leq 1$. Suppose that $f:I\subseteq \left( 0,\infty
\right) \mathbb{\rightarrow }%
\mathbb{R}
$ is GA-$s$-convex function in the first sense and $a,b\in I$ with $a<b$. If 
$f\in L\left[ a,b\right] $, then one has the inequalities:%
\begin{equation}
f\left( \sqrt{ab}\right) \leq \frac{1}{\ln b-\ln a}\dint\limits_{a}^{b}\frac{%
f(x)}{x}dx\leq \frac{f(a)+sf(b)}{s+1}  \label{3-1}
\end{equation}
\end{theorem}

\begin{proof}
As $f$ is GA-$s$-convex function in the first sense, we have, for all $%
x,y\in I$%
\begin{equation}
f\left( \sqrt{xy}\right) \leq \frac{1}{2^{s}}f(x)+\left( 1-\frac{1}{2^{s}}%
\right) f(y).  \label{3-1a}
\end{equation}%
Now, let $x=a^{1-t}b^{t}$ and $y=a^{t}b^{1-t}$ with $t\in \left[ 0,1\right] $%
. Then we get by (\ref{3-1a}) that:%
\begin{equation*}
f\left( \sqrt{ab}\right) \leq \frac{1}{2^{s}}f(a^{1-t}b^{t})+\left( 1-\frac{1%
}{2^{s}}\right) f(a^{t}b^{1-t})
\end{equation*}%
for all $t\in \left[ 0,1\right] $. Integrating this inequality on $\left[ 0,1%
\right] $, we deduce the first part of (\ref{3-1}).

Secondly, we observe that for all $t\in \left[ 0,1\right] $%
\begin{equation*}
f(a^{t}b^{1-t})\leq t^{s}f(a)+(1-t^{s})f(b).
\end{equation*}%
Integrating this inequality on $\left[ 0,1\right] $, we get%
\begin{equation*}
\dint\limits_{0}^{1}f(a^{t}b^{1-t})dt\leq \frac{f(a)+sf(b)}{s+1}.
\end{equation*}%
As the change of variable $x=a^{t}b^{1-t}$ gives us that%
\begin{equation*}
\dint\limits_{0}^{1}f(a^{t}b^{1-t})dt=\frac{1}{\ln b-\ln a}%
\dint\limits_{a}^{b}\frac{f(x)}{x}dx,
\end{equation*}%
the second inequality in (\ref{3-1}) is proved.
\end{proof}

Similarly to Theorem \ref{3.1}, we will give the following theorem for GA-$s$%
-convex function in the second sense:

\begin{theorem}
Suppose that $f:I\subseteq \left( 0,\infty \right) \mathbb{\rightarrow }%
\mathbb{R}
$ is GA-$s$-convex function in the second sense and $a,b\in I$ with $a<b$.
If $f\in L\left[ a,b\right] $, then one has the inequalities:%
\begin{equation}
2^{s-1}f\left( \sqrt{ab}\right) \leq \frac{1}{\ln b-\ln a}%
\dint\limits_{a}^{b}\frac{f(x)}{x}dx\leq \frac{f(a)+f(b)}{s+1}  \label{3-2}
\end{equation}
\end{theorem}

\begin{proof}
As $f$ is GA-$s$-convex function in the second sense, we have, for all $%
x,y\in I$%
\begin{equation}
f\left( \sqrt{xy}\right) \leq \frac{f(x)+f(y)}{2^{s}}.  \label{3-2a}
\end{equation}%
Now, let $x=a^{1-t}b^{t}$ and $y=a^{t}b^{1-t}$ with $t\in \left[ 0,1\right] $%
. Then we get by (\ref{3-2a}) that:%
\begin{equation*}
f\left( \sqrt{ab}\right) \leq \frac{f(a^{1-t}b^{t})+f(a^{t}b^{1-t})}{2^{s}}
\end{equation*}%
for all $t\in \left[ 0,1\right] $. Integrating this inequality on $\left[ 0,1%
\right] $, we deduce the first part of (\ref{3-2}).

Secondly, we observe that for all $t\in \left[ 0,1\right] $%
\begin{equation*}
f(a^{t}b^{1-t})\leq t^{s}f(a)+(1-t)^{s}f(b).
\end{equation*}%
Integrating this inequality on $\left[ 0,1\right] $, we get%
\begin{equation*}
\dint\limits_{0}^{1}f(a^{t}b^{1-t})dt\leq \frac{f(a)+f(b)}{s+1}.
\end{equation*}%
As the change of variable $x=a^{t}b^{1-t}$ gives us that%
\begin{equation*}
\dint\limits_{0}^{1}f(a^{t}b^{1-t})dt=\frac{1}{\ln b-\ln a}%
\dint\limits_{a}^{b}\frac{f(x)}{x}dx,
\end{equation*}%
the second inequality in (\ref{3-2}) is proved.
\end{proof}

\begin{remark}
The constant $k=1/(s+1)$ for $s\in \left( 0,1\right] $ is the best possible
in the second inequality in (\ref{3-2}). Indeed, as the mapping $f:\left[ a,b%
\right] \rightarrow \left[ a,b\right] $ given $f(x)=s+1$, $0<a<b$, is GA-$s$%
-convex in the second sense and 
\begin{equation*}
\frac{1}{\ln b-\ln a}\dint\limits_{a}^{b}\frac{f(x)}{x}dx=s+1=\frac{f(a)+f(b)%
}{s+1}
\end{equation*}
\end{remark}

\begin{theorem}
\label{2.1}Let $f:I\subseteq \left( 0,\infty \right) \mathbb{\rightarrow R}$
be differentiable on $I^{\circ }$, and $a,b\in I^{\circ }$ with $a<b$ and $%
f^{\prime }\in L\left[ a,b\right] .$

a) If $\left\vert f^{\prime }\right\vert ^{q}$ is GA-$s$-convex function in
the second sense on $\left[ a,b\right] $ for $q\geq 1$ and $s\in \left( 0,1%
\right] ,$ then%
\begin{equation}
\left\vert \frac{f(a)+f(b)}{2}-\frac{1}{\ln b-\ln a}\dint\limits_{a}^{b}%
\frac{f(x)}{x}dx\right\vert   \label{3-3}
\end{equation}%
\begin{eqnarray*}
&\leq &\ln \left( \frac{b}{a}\right) \left( \frac{1}{2}\right) ^{2-\frac{1}{q%
}}\left[ a\left\{ c_{1}(s,q)\left\vert f^{\prime }\left( a\right)
\right\vert ^{q}+c_{2}(s,q)\left\vert f^{\prime }\left( b\right) \right\vert
^{q}\right\} ^{\frac{1}{q}}\right.  \\
&&\left. +b\left\{ c_{3}(s,q)\left\vert f^{\prime }\left( b\right)
\right\vert ^{q}+c_{4}(s,q)\left\vert f^{\prime }\left( a\right) \right\vert
^{q}\right\} ^{\frac{1}{q}}\right] 
\end{eqnarray*}%
\begin{equation}
\left\vert f\left( \sqrt{ab}\right) -\frac{1}{\ln b-\ln a}%
\dint\limits_{a}^{b}\frac{f(x)}{x}dx\right\vert   \label{3-4}
\end{equation}%
\begin{eqnarray*}
&\leq &\ln \left( \frac{b}{a}\right) \left( \frac{1}{2}\right) ^{3-\frac{1}{q%
}}\left[ a\left\{ c_{1}(s,q/2)\left\vert f^{\prime }\left( a\right)
\right\vert ^{q}+c_{2}(s,q/2)\left\vert f^{\prime }\left( \sqrt{ab}\right)
\right\vert ^{q}\right\} ^{\frac{1}{q}}\right.  \\
&&\left. +b\left\{ c_{3}(s,q/2)\left\vert f^{\prime }\left( b\right)
\right\vert ^{q}+c_{4}(s,q/2)\left\vert f^{\prime }\left( \sqrt{ab}\right)
\right\vert ^{q}\right\} ^{\frac{1}{q}}\right] 
\end{eqnarray*}%
where 
\begin{eqnarray}
c_{1}(s,q) &=&\dint\limits_{0}^{1}t\left( 1-t\right) ^{s}\left( \frac{b}{a}%
\right) ^{qt}dt,\ c_{2}(s,q)=\dint\limits_{0}^{1}t^{s+1}\left( \frac{b}{a}%
\right) ^{qt}dt,  \label{3-41} \\
c_{3}(s,q) &=&\dint\limits_{0}^{1}t\left( 1-t\right) ^{s}\left( \frac{a}{b}%
\right) ^{qt}dt,\ c_{4}(s,q)=\dint\limits_{0}^{1}t^{s+1}\left( \frac{a}{b}%
\right) ^{qt}dt,  \notag
\end{eqnarray}%
b) If $\left\vert f^{\prime }\right\vert ^{q}$ is GA-$s$-convex function in
the first sense on $\left[ a,b\right] $ for $q\geq 1$ and $s\in \left( 0,1%
\right] ,$ then%
\begin{equation}
\left\vert \frac{f(a)+f(b)}{2}-\frac{1}{\ln b-\ln a}\dint\limits_{a}^{b}%
\frac{f(x)}{x}dx\right\vert   \label{3-5}
\end{equation}%
\begin{eqnarray*}
&\leq &\ln \left( \frac{b}{a}\right) \left( \frac{1}{2}\right) ^{2-\frac{1}{q%
}}\left[ a\left\{ c_{5}(s,q)\left\vert f^{\prime }\left( a\right)
\right\vert ^{q}+c_{2}(s,q)\left\vert f^{\prime }\left( b\right) \right\vert
^{q}\right\} ^{\frac{1}{q}}\right.  \\
&&\left. +b\left\{ c_{6}(s,q)\left\vert f^{\prime }\left( b\right)
\right\vert ^{q}+c_{4}(s,q)\left\vert f^{\prime }\left( a\right) \right\vert
^{q}\right\} ^{\frac{1}{q}}\right] 
\end{eqnarray*}%
\begin{equation}
\left\vert f\left( \sqrt{ab}\right) -\frac{1}{\ln b-\ln a}%
\dint\limits_{a}^{b}\frac{f(x)}{x}dx\right\vert   \label{3-6}
\end{equation}%
\begin{eqnarray*}
&\leq &\ln \left( \frac{b}{a}\right) \left( \frac{1}{2}\right) ^{3-\frac{1}{q%
}}\left[ a\left\{ c_{5}(s,q/2)\left\vert f^{\prime }\left( a\right)
\right\vert ^{q}+c_{2}(s,q/2)\left\vert f^{\prime }\left( \sqrt{ab}\right)
\right\vert ^{q}\right\} ^{\frac{1}{q}}\right.  \\
&&\left. +b\left\{ c_{6}(s,q/2)\left\vert f^{\prime }\left( b\right)
\right\vert ^{q}+c_{4}(s,q/2)\left\vert f^{\prime }\left( \sqrt{ab}\right)
\right\vert ^{q}\right\} ^{\frac{1}{q}}\right] ,
\end{eqnarray*}%
where%
\begin{equation}
c_{5}(s,q)=\dint\limits_{0}^{1}t\left( 1-t^{s}\right) \left( \frac{b}{a}%
\right) ^{qt}dt,\ c_{6}(s,q)=\dint\limits_{0}^{1}t\left( 1-t^{s}\right)
\left( \frac{a}{b}\right) ^{qt}dt.  \label{3-61}
\end{equation}
\end{theorem}

\begin{proof}
a) (1) Since $\left\vert f^{\prime }\right\vert ^{q}$ is GA-$s$-convex
function in the second sense on $\left[ a,b\right] $, from lemma \ref{1.1}
and power mean inequality, we have%
\begin{eqnarray*}
&&\left\vert \frac{f(a)+f(b)}{2}-\frac{1}{\ln b-\ln a}\dint\limits_{a}^{b}%
\frac{f(x)}{x}dx\right\vert  \\
&\leq &\frac{\ln \left( \frac{b}{a}\right) }{2}\left[ a\dint\limits_{0}^{1}t%
\left( \frac{b}{a}\right) ^{t}\left\vert f^{\prime }\left(
a^{1-t}b^{t}\right) \right\vert dt+b\dint\limits_{0}^{1}t\left( \frac{a}{b}%
\right) ^{t}\left\vert f^{\prime }\left( b^{1-t}a^{t}\right) \right\vert dt%
\right] 
\end{eqnarray*}%
\begin{eqnarray}
&\leq &\frac{a\ln \left( \frac{b}{a}\right) }{2}\left(
\dint\limits_{0}^{1}tdt\right) ^{1-\frac{1}{q}}\left(
\dint\limits_{0}^{1}t\left( \frac{b}{a}\right) ^{qt}\left\vert f^{\prime
}\left( a^{1-t}b^{t}\right) \right\vert ^{q}dt\right) ^{\frac{1}{q}}  \notag
\\
&&+\frac{b\ln \left( \frac{b}{a}\right) }{2}\left(
\dint\limits_{0}^{1}tdt\right) ^{1-\frac{1}{q}}\left(
\dint\limits_{0}^{1}t\left( \frac{a}{b}\right) ^{qt}\left\vert f^{\prime
}\left( b^{1-t}a^{t}\right) \right\vert ^{q}dt\right) ^{\frac{1}{q}}
\label{3-3a}
\end{eqnarray}%
\begin{eqnarray*}
&\leq &\frac{a\ln \left( \frac{b}{a}\right) }{2}\left( \frac{1}{2}\right)
^{1-\frac{1}{q}}\left( \dint\limits_{0}^{1}t\left( \frac{b}{a}\right)
^{qt}\left( \left( 1-t\right) ^{s}\left\vert f^{\prime }\left( a\right)
\right\vert ^{q}+t^{s}\left\vert f^{\prime }\left( b\right) \right\vert
^{q}\right) dt\right) ^{\frac{1}{q}} \\
&&+\frac{b\ln \left( \frac{b}{a}\right) }{2}\left( \frac{1}{2}\right) ^{1-%
\frac{1}{q}}\left( \dint\limits_{0}^{1}t\left( \frac{a}{b}\right)
^{qt}\left( \left( 1-t\right) ^{s}\left\vert f^{\prime }\left( b\right)
\right\vert ^{q}+t^{s}\left\vert f^{\prime }\left( a\right) \right\vert
^{q}\right) dt\right) ^{\frac{1}{q}}
\end{eqnarray*}%
\begin{eqnarray*}
&\leq &a\ln \left( \frac{b}{a}\right) \left( \frac{1}{2}\right) ^{2-\frac{1}{%
q}}\left\{ c_{1}(s,q)\left\vert f^{\prime }\left( a\right) \right\vert
^{q}+c_{2}(s,q)\left\vert f^{\prime }\left( b\right) \right\vert
^{q}\right\} ^{\frac{1}{q}} \\
&&+b\ln \left( \frac{b}{a}\right) \left( \frac{1}{2}\right) ^{2-\frac{1}{q}%
}\left\{ c_{3}(s,q)\left\vert f^{\prime }\left( b\right) \right\vert
^{q}+c_{4}(s,q)\left\vert f^{\prime }\left( a\right) \right\vert
^{q}\right\} ^{\frac{1}{q}}.
\end{eqnarray*}

(2) Since $\left\vert f^{\prime }\right\vert ^{q}$ is GA-$s$-convex function
in the second sense on $\left[ a,b\right] $, from lemma \ref{2.1} and power
mean inequality, we have%
\begin{eqnarray*}
&&\left\vert f\left( \sqrt{ab}\right) -\frac{1}{\ln b-\ln a}%
\dint\limits_{a}^{b}\frac{f(x)}{x}dx\right\vert  \\
&\leq &\frac{\ln \frac{b}{a}}{4}\left[ a\dint\limits_{0}^{1}t\left( \frac{b}{%
a}\right) ^{\frac{t}{2}}\left\vert f^{\prime }\left( a^{1-t}\left( ab\right)
^{\frac{t}{2}}\right) \right\vert dt+b\dint\limits_{0}^{1}t\left( \frac{a}{b}%
\right) ^{\frac{t}{2}}\left\vert f^{\prime }\left( b^{1-t}\left( ab\right) ^{%
\frac{t}{2}}\right) \right\vert dt\right] 
\end{eqnarray*}%
\begin{eqnarray}
&\leq &\frac{a\ln \frac{b}{a}}{4}\left( \dint\limits_{0}^{1}tdt\right) ^{1-%
\frac{1}{q}}\left( \dint\limits_{0}^{1}t\left( \frac{b}{a}\right) ^{\frac{qt%
}{2}}\left\vert f^{\prime }\left( a^{1-t}\left( ab\right) ^{\frac{t}{2}%
}\right) \right\vert ^{q}dt\right) ^{\frac{1}{q}}  \notag \\
&&+\frac{b\ln \frac{b}{a}}{4}\left( \dint\limits_{0}^{1}tdt\right) ^{1-\frac{%
1}{q}}\left( \dint\limits_{0}^{1}t\left( \frac{a}{b}\right) ^{\frac{qt}{2}%
}\left\vert f^{\prime }\left( b^{1-t}\left( ab\right) ^{\frac{t}{2}}\right)
\right\vert ^{q}dt\right) ^{\frac{1}{q}}  \label{3-4a}
\end{eqnarray}%
\begin{eqnarray*}
&\leq &\frac{a\ln \frac{b}{a}}{4}\left( \frac{1}{2}\right) ^{1-\frac{1}{q}%
}\left( \dint\limits_{0}^{1}t\left( \frac{b}{a}\right) ^{\frac{qt}{2}}\left(
\left( 1-t\right) ^{s}\left\vert f^{\prime }\left( a\right) \right\vert
^{q}+t^{s}\left\vert f^{\prime }\left( \sqrt{ab}\right) \right\vert
^{q}\right) dt\right) ^{\frac{1}{q}} \\
&&+\frac{b\ln \frac{b}{a}}{4}\left( \frac{1}{2}\right) ^{1-\frac{1}{q}%
}\left( \dint\limits_{0}^{1}t\left( \frac{a}{b}\right) ^{\frac{qt}{2}}\left(
\left( 1-t\right) ^{s}\left\vert f^{\prime }\left( b\right) \right\vert
^{q}+t^{s}\left\vert f^{\prime }\left( \sqrt{ab}\right) \right\vert
^{q}\right) dt\right) ^{\frac{1}{q}}
\end{eqnarray*}%
\begin{eqnarray*}
&\leq &\ln \left( \frac{b}{a}\right) \left( \frac{1}{2}\right) ^{3-\frac{1}{q%
}}\left[ a\left\{ c_{1}(s,q/2)\left\vert f^{\prime }\left( a\right)
\right\vert ^{q}+c_{2}(s,q/2)\left\vert f^{\prime }\left( \sqrt{ab}\right)
\right\vert ^{q}\right\} ^{\frac{1}{q}}\right.  \\
&&\left. +b\left\{ c_{5}(s,q/2)\left\vert f^{\prime }\left( b\right)
\right\vert ^{q}+c_{6}(s,q/2)\left\vert f^{\prime }\left( \sqrt{ab}\right)
\right\vert ^{q}\right\} ^{\frac{1}{q}}\right] ,
\end{eqnarray*}%
b) (1) Since $\left\vert f^{\prime }\right\vert ^{q}$ is GA-$s$-convex
function in the first sense on $\left[ a,b\right] $, from the inequality (%
\ref{3-3a}), we have%
\begin{equation*}
\left\vert \frac{f(a)+f(b)}{2}-\frac{1}{\ln b-\ln a}\dint\limits_{a}^{b}%
\frac{f(x)}{x}dx\right\vert 
\end{equation*}%
\begin{eqnarray*}
&\leq &\frac{a\ln \left( \frac{b}{a}\right) }{2}\left( \frac{1}{2}\right)
^{1-\frac{1}{q}}\left( \dint\limits_{0}^{1}t\left( \frac{b}{a}\right)
^{qt}\left( \left( 1-t^{s}\right) \left\vert f^{\prime }\left( a\right)
\right\vert ^{q}+t^{s}\left\vert f^{\prime }\left( b\right) \right\vert
^{q}\right) dt\right) ^{\frac{1}{q}} \\
&&+\frac{b\ln \left( \frac{b}{a}\right) }{2}\left( \frac{1}{2}\right) ^{1-%
\frac{1}{q}}\left( \dint\limits_{0}^{1}t\left( \frac{a}{b}\right)
^{qt}\left( \left( 1-t^{s}\right) \left\vert f^{\prime }\left( b\right)
\right\vert ^{q}+t^{s}\left\vert f^{\prime }\left( a\right) \right\vert
^{q}\right) dt\right) ^{\frac{1}{q}}
\end{eqnarray*}%
\begin{eqnarray*}
&\leq &a\ln \left( \frac{b}{a}\right) \left( \frac{1}{2}\right) ^{2-\frac{1}{%
q}}\left\{ c_{5}(s,q)\left\vert f^{\prime }\left( a\right) \right\vert
^{q}+c_{2}(s,q)\left\vert f^{\prime }\left( b\right) \right\vert
^{q}\right\} ^{\frac{1}{q}} \\
&&+b\ln \left( \frac{b}{a}\right) \left( \frac{1}{2}\right) ^{2-\frac{1}{q}%
}\left\{ c_{6}(s,q)\left\vert f^{\prime }\left( b\right) \right\vert
^{q}+c_{4}(s,q)\left\vert f^{\prime }\left( a\right) \right\vert
^{q}\right\} ^{\frac{1}{q}}.
\end{eqnarray*}

(2) Since $\left\vert f^{\prime }\right\vert ^{q}$ is GA-$s$-convex function
in the first sense on $\left[ a,b\right] $, the inequality (\ref{3-6}) is
easily obtained by using the inequality (\ref{3-4a}).
\end{proof}

If taking $s=1$ in Theorem \ref{2.1}, we can derive the following
inequalities for GA-convex.

\begin{corollary}
\label{2.1a}Let $f:I\subseteq \mathbb{R}_{+}\mathbb{\rightarrow R}$ be
differentiable on $I^{\circ }$, and $a,b\in I^{\circ }$ with $a<b$ and $%
f^{\prime }\in L\left[ a,b\right] .$ If $\left\vert f^{\prime }\right\vert
^{q}$ is GA-convex on $\left[ a,b\right] $ for $q\geq 1$, then%
\begin{eqnarray}
&&\left\vert \frac{f(a)+f(b)}{2}-\frac{1}{\ln b-\ln a}\dint\limits_{a}^{b}%
\frac{f(x)}{x}dx\right\vert  \label{3-7a} \\
&\leq &\ln \left( \frac{b}{a}\right) \left( \frac{1}{2}\right) ^{2-\frac{1}{q%
}}\left[ a\left\{ c_{1}(1,q)\left\vert f^{\prime }\left( a\right)
\right\vert ^{q}+c_{2}(1,q)\left\vert f^{\prime }\left( b\right) \right\vert
^{q}\right\} ^{\frac{1}{q}}\right.  \notag \\
&&\left. +b\left\{ c_{3}(1,q)\left\vert f^{\prime }\left( b\right)
\right\vert ^{q}+c_{4}(1,q)\left\vert f^{\prime }\left( a\right) \right\vert
^{q}\right\} ^{\frac{1}{q}}\right] ,  \notag
\end{eqnarray}%
\begin{eqnarray}
&&\left\vert f\left( \sqrt{ab}\right) -\frac{1}{\ln b-\ln a}%
\dint\limits_{a}^{b}\frac{f(x)}{x}dx\right\vert  \label{3-7b} \\
&\leq &\ln \left( \frac{b}{a}\right) \left( \frac{1}{2}\right) ^{3-\frac{1}{q%
}}\left[ a\left\{ c_{1}(1,q/2)\left\vert f^{\prime }\left( a\right)
\right\vert ^{q}+c_{2}(1,q/2)\left\vert f^{\prime }\left( \sqrt{ab}\right)
\right\vert ^{q}\right\} ^{\frac{1}{q}}\right.  \notag \\
&&\left. +b\left\{ c_{3}(1,q/2)\left\vert f^{\prime }\left( b\right)
\right\vert ^{q}+c_{4}(1,q/2)\left\vert f^{\prime }\left( \sqrt{ab}\right)
\right\vert ^{q}\right\} ^{\frac{1}{q}}\right]  \notag
\end{eqnarray}
\end{corollary}

If taking $q=1$ in Theorem \ref{2.1}, we can derive the following corollary.

\begin{corollary}
Let $f:I\subseteq \left( 0,\infty \right) \mathbb{\rightarrow R}$ be
differentiable on $I^{\circ }$, and $a,b\in I^{\circ }$ with $a<b$ and $%
f^{\prime }\in L\left[ a,b\right] .$

a) If $\left\vert f^{\prime }\right\vert $ is GA-$s$-convex function in the
second sense on $\left[ a,b\right] $, $s\in \left( 0,1\right] $, then%
\begin{equation*}
\left\vert \frac{f(a)+f(b)}{2}-\frac{1}{\ln b-\ln a}\dint\limits_{a}^{b}%
\frac{f(x)}{x}dx\right\vert 
\end{equation*}%
\begin{eqnarray*}
&\leq &\frac{\ln \left( \frac{b}{a}\right) }{2}\left[ \left(
ac_{1}(s,1)+bc_{4}(s,1)\right) \left\vert f^{\prime }\left( a\right)
\right\vert \right.  \\
&&\left. +\left( bc_{3}(s,1)+ac_{2}(s,1)\right) \left\vert f^{\prime }\left(
b\right) \right\vert \right] 
\end{eqnarray*}%
\begin{equation*}
\left\vert f\left( \sqrt{ab}\right) -\frac{1}{\ln b-\ln a}%
\dint\limits_{a}^{b}\frac{f(x)}{x}dx\right\vert 
\end{equation*}%
\begin{eqnarray*}
&\leq &\frac{\ln \left( \frac{b}{a}\right) }{4}\left[ ac_{1}(s,1/2)\left%
\vert f^{\prime }\left( a\right) \right\vert +bc_{3}(s,1/2)\left\vert
f^{\prime }\left( b\right) \right\vert \right.  \\
&&\left. +\left( ac_{2}(s,1/2)+bc_{4}(s,1/2)\right) \left\vert f^{\prime
}\left( \sqrt{ab}\right) \right\vert \right] 
\end{eqnarray*}%
where $c_{1}$, $c_{2}$, $c_{3}$, $c_{4}$ are defined by (\ref{3-41}).

b) If $\left\vert f^{\prime }\right\vert ^{q}$ is GA-$s$-convex function in
the first sense on $\left[ a,b\right] $, $s\in \left( 0,1\right] $, then%
\begin{equation*}
\left\vert \frac{f(a)+f(b)}{2}-\frac{1}{\ln b-\ln a}\dint\limits_{a}^{b}%
\frac{f(x)}{x}dx\right\vert 
\end{equation*}%
\begin{eqnarray*}
&\leq &\frac{\ln \left( \frac{b}{a}\right) }{2}\left[ \left(
ac_{5}(s,1)+bc_{4}(s,1)\right) \left\vert f^{\prime }\left( a\right)
\right\vert \right.  \\
&&\left. +\left( ac_{2}(s,1)+bc_{6}(s,1)\right) \left\vert f^{\prime }\left(
b\right) \right\vert \right] ,
\end{eqnarray*}%
\begin{equation*}
\left\vert f\left( \sqrt{ab}\right) -\frac{1}{\ln b-\ln a}%
\dint\limits_{a}^{b}\frac{f(x)}{x}dx\right\vert 
\end{equation*}%
\begin{eqnarray*}
&\leq &\frac{\ln \left( \frac{b}{a}\right) }{4}\left[ ac_{5}(s,1/2)\left%
\vert f^{\prime }\left( a\right) \right\vert +bc_{6}(s,1/2)\left\vert
f^{\prime }\left( b\right) \right\vert \right.  \\
&&\left. +\left( ac_{2}(s,1/2)+bc_{4}(s,1/2)\right) \left\vert f^{\prime
}\left( \sqrt{ab}\right) \right\vert \right] ,
\end{eqnarray*}%
where $c_{2}$, $c_{4}$, $c_{5}$, $c_{6}$ are defined by (\ref{3-41}) and (%
\ref{3-61}).
\end{corollary}

\begin{theorem}
\label{2.2}Let $f:I\subseteq \left( 0,\infty \right) \mathbb{\rightarrow R}$
be differentiable on $I^{\circ }$, and $a,b\in I^{\circ }$ with $a<b$ and $%
f^{\prime }\in L\left[ a,b\right] .$

a) If $\left\vert f^{\prime }\right\vert ^{q}$ is GA-$s$-convex function in
the second sense on $\left[ a,b\right] $ for $q>1$ and $s\in \left( 0,1%
\right] ,$ then%
\begin{equation}
\left\vert \frac{f(a)+f(b)}{2}-\frac{1}{\ln b-\ln a}\dint\limits_{a}^{b}%
\frac{f(x)}{x}dx\right\vert   \label{3-8}
\end{equation}%
\begin{eqnarray*}
&\leq &\frac{\ln \left( \frac{b}{a}\right) }{2}\left( \frac{q-1}{2q-1}%
\right) ^{1-\frac{1}{q}}\left[ a\left\{ c_{7}(s,q)\left\vert f^{\prime
}\left( a\right) \right\vert ^{q}+c_{8}(s,q)\left\vert f^{\prime }\left(
b\right) \right\vert ^{q}\right\} ^{\frac{1}{q}}\right.  \\
&&\left. +b\left\{ c_{9}(s,q)\left\vert f^{\prime }\left( b\right)
\right\vert ^{q}+c_{10}(s,q)\left\vert f^{\prime }\left( a\right)
\right\vert ^{q}\right\} ^{\frac{1}{q}}\right] 
\end{eqnarray*}%
\begin{equation}
\left\vert f\left( \sqrt{ab}\right) -\frac{1}{\ln b-\ln a}%
\dint\limits_{a}^{b}\frac{f(x)}{x}dx\right\vert   \label{3-9}
\end{equation}%
\begin{eqnarray*}
&\leq &\frac{\ln \left( \frac{b}{a}\right) }{4}\left( \frac{q-1}{2q-1}%
\right) ^{1-\frac{1}{q}}\left[ a\left\{ c_{7}(s,q/2)\left\vert f^{\prime
}\left( a\right) \right\vert ^{q}+c_{8}(s,q/2)\left\vert f^{\prime }\left( 
\sqrt{ab}\right) \right\vert ^{q}\right\} ^{\frac{1}{q}}\right.  \\
&&\left. +b\left\{ c_{9}(s,q/2)\left\vert f^{\prime }\left( b\right)
\right\vert ^{q}+c_{10}(s,q/2)\left\vert f^{\prime }\left( \sqrt{ab}\right)
\right\vert ^{q}\right\} ^{\frac{1}{q}}\right] 
\end{eqnarray*}%
where 
\begin{eqnarray}
c_{7}(s,q) &=&\dint\limits_{0}^{1}\left( 1-t\right) ^{s}\left( \frac{b}{a}%
\right) ^{qt}dt,\ c_{8}(s,q)=\dint\limits_{0}^{1}t^{s}\left( \frac{b}{a}%
\right) ^{qt}dt,  \label{3-91} \\
c_{9}(s,q) &=&\dint\limits_{0}^{1}\left( 1-t\right) ^{s}\left( \frac{a}{b}%
\right) ^{qt}dt,\ c_{10}(s,q)=\dint\limits_{0}^{1}t^{s}\left( \frac{a}{b}%
\right) ^{qt}dt,  \notag
\end{eqnarray}%
b) If $\left\vert f^{\prime }\right\vert ^{q}$ is GA-$s$-convex function in
the first sense on $\left[ a,b\right] $ for $q>1$ and $s\in \left( 0,1\right]
,$ then%
\begin{equation}
\left\vert \frac{f(a)+f(b)}{2}-\frac{1}{\ln b-\ln a}\dint\limits_{a}^{b}%
\frac{f(x)}{x}dx\right\vert   \label{3-10}
\end{equation}%
\begin{eqnarray*}
&\leq &\frac{\ln \left( \frac{b}{a}\right) }{2}\left( \frac{q-1}{2q-1}%
\right) ^{1-\frac{1}{q}}\left[ a\left\{ c_{11}(s,q)\left\vert f^{\prime
}\left( a\right) \right\vert ^{q}+c_{8}(s,q)\left\vert f^{\prime }\left(
b\right) \right\vert ^{q}\right\} ^{\frac{1}{q}}\right.  \\
&&\left. +b\left\{ c_{12}(s,q)\left\vert f^{\prime }\left( b\right)
\right\vert ^{q}+c_{10}(s,q)\left\vert f^{\prime }\left( a\right)
\right\vert ^{q}\right\} ^{\frac{1}{q}}\right] 
\end{eqnarray*}%
\begin{equation}
\left\vert f\left( \sqrt{ab}\right) -\frac{1}{\ln b-\ln a}%
\dint\limits_{a}^{b}\frac{f(x)}{x}dx\right\vert   \label{3-11}
\end{equation}%
\begin{eqnarray*}
&\leq &\frac{\ln \left( \frac{b}{a}\right) }{4}\left( \frac{q-1}{2q-1}%
\right) ^{1-\frac{1}{q}}\left[ a\left\{ c_{11}(s,q/2)\left\vert f^{\prime
}\left( a\right) \right\vert ^{q}+c_{8}(s,q/2)\left\vert f^{\prime }\left( 
\sqrt{ab}\right) \right\vert ^{q}\right\} ^{\frac{1}{q}}\right.  \\
&&\left. +b\left\{ c_{12}(s,q/2)\left\vert f^{\prime }\left( b\right)
\right\vert ^{q}+c_{10}(s,q/2)\left\vert f^{\prime }\left( \sqrt{ab}\right)
\right\vert ^{q}\right\} ^{\frac{1}{q}}\right] ,
\end{eqnarray*}%
where%
\begin{equation}
c_{11}(s,q)=\dint\limits_{0}^{1}\left( 1-t^{s}\right) \left( \frac{b}{a}%
\right) ^{qt}dt,\ c_{12}(s,q)=\dint\limits_{0}^{1}\left( 1-t^{s}\right)
\left( \frac{a}{b}\right) ^{qt}dt.  \label{3-101}
\end{equation}
\end{theorem}

\begin{proof}
a) (1) Since $\left\vert f^{\prime }\right\vert ^{q}$ is GA-$s$-convex
function in the second sense on $\left[ a,b\right] $, from lemma \ref{1.1}
and H\"{o}lder inequality, we have%
\begin{eqnarray*}
&&\left\vert \frac{f(a)+f(b)}{2}-\frac{1}{\ln b-\ln a}\dint\limits_{a}^{b}%
\frac{f(x)}{x}dx\right\vert  \\
&\leq &\frac{\ln \left( \frac{b}{a}\right) }{2}\left[ a\dint\limits_{0}^{1}t%
\left( \frac{b}{a}\right) ^{t}\left\vert f^{\prime }\left(
a^{1-t}b^{t}\right) \right\vert dt+b\dint\limits_{0}^{1}t\left( \frac{a}{b}%
\right) ^{t}\left\vert f^{\prime }\left( b^{1-t}a^{t}\right) \right\vert dt%
\right] 
\end{eqnarray*}%
\begin{eqnarray}
&\leq &\frac{a\ln \left( \frac{b}{a}\right) }{2}\left(
\dint\limits_{0}^{1}t^{\frac{q}{q-1}}dt\right) ^{1-\frac{1}{q}}\left(
\dint\limits_{0}^{1}\left( \frac{b}{a}\right) ^{qt}\left\vert f^{\prime
}\left( a^{1-t}b^{t}\right) \right\vert ^{q}dt\right) ^{\frac{1}{q}}  \notag
\\
&&+\frac{b}{2}\ln \left( \frac{b}{a}\right) \left( \dint\limits_{0}^{1}t^{%
\frac{q}{q-1}}dt\right) ^{1-\frac{1}{q}}\left( \dint\limits_{0}^{1}\left( 
\frac{a}{b}\right) ^{qt}\left\vert f^{\prime }\left( b^{1-t}a^{t}\right)
\right\vert ^{q}dt\right) ^{\frac{1}{q}}  \label{3-8a}
\end{eqnarray}%
\begin{eqnarray*}
&\leq &\frac{a\ln \left( \frac{b}{a}\right) }{2}\left( \frac{q-1}{2q-1}%
\right) ^{1-\frac{1}{q}}\left( \dint\limits_{0}^{1}\left( \frac{b}{a}\right)
^{qt}\left( \left( 1-t\right) ^{s}\left\vert f^{\prime }\left( a\right)
\right\vert ^{q}+t^{s}\left\vert f^{\prime }\left( b\right) \right\vert
^{q}\right) dt\right) ^{\frac{1}{q}} \\
&&+\frac{b}{2}\ln \left( \frac{b}{a}\right) \left( \frac{q-1}{2q-1}\right)
^{1-\frac{1}{q}}\left( \dint\limits_{0}^{1}\left( \frac{a}{b}\right)
^{qt}\left( \left( 1-t\right) ^{s}\left\vert f^{\prime }\left( b\right)
\right\vert ^{q}+t^{s}\left\vert f^{\prime }\left( a\right) \right\vert
^{q}\right) dt\right) ^{\frac{1}{q}}
\end{eqnarray*}%
\begin{eqnarray*}
&\leq &\frac{\ln \left( \frac{b}{a}\right) }{2}\left( \frac{q-1}{2q-1}%
\right) ^{1-\frac{1}{q}}\left[ a\left\{ c_{7}(s,q)\left\vert f^{\prime
}\left( a\right) \right\vert ^{q}+c_{8}(s,q)\left\vert f^{\prime }\left(
b\right) \right\vert ^{q}\right\} ^{\frac{1}{q}}\right.  \\
&&\left. +b\left\{ c_{9}(s,q)\left\vert f^{\prime }\left( b\right)
\right\vert ^{q}+c_{10}(s,q)\left\vert f^{\prime }\left( a\right)
\right\vert ^{q}\right\} ^{\frac{1}{q}}\right] .
\end{eqnarray*}

(2)Since $\left\vert f^{\prime }\right\vert ^{q}$ is GA-$s$-convex function
in the first sense on $\left[ a,b\right] $, from lemma \ref{1.1} and H\"{o}%
lder inequality, we have%
\begin{eqnarray}
&&\left\vert f\left( \sqrt{ab}\right) -\frac{1}{\ln b-\ln a}%
\dint\limits_{a}^{b}\frac{f(x)}{x}dx\right\vert   \notag \\
&\leq &\frac{\ln \frac{b}{a}}{4}\left[ a\dint\limits_{0}^{1}t\left( \frac{b}{%
a}\right) ^{\frac{t}{2}}\left\vert f^{\prime }\left( a^{1-t}\left( ab\right)
^{\frac{t}{2}}\right) \right\vert dt+b\dint\limits_{0}^{1}t\left( \frac{a}{b}%
\right) ^{\frac{t}{2}}\left\vert f^{\prime }\left( b^{1-t}\left( ab\right) ^{%
\frac{t}{2}}\right) \right\vert dt\right]   \label{3-9a}
\end{eqnarray}%
\begin{eqnarray*}
&\leq &\frac{a\ln \frac{b}{a}}{4}\left( \frac{q-1}{2q-1}\right) ^{1-\frac{1}{%
q}}\left( \dint\limits_{0}^{1}\left( \frac{b}{a}\right) ^{\frac{qt}{2}%
}\left( \left( 1-t\right) ^{s}\left\vert f^{\prime }\left( a\right)
\right\vert ^{q}+t^{s}\left\vert f^{\prime }\left( \sqrt{ab}\right)
\right\vert ^{q}\right) dt\right) ^{\frac{1}{q}} \\
&&+\frac{b\ln \frac{b}{a}}{4}\left( \frac{q-1}{2q-1}\right) ^{1-\frac{1}{q}%
}\left( \dint\limits_{0}^{1}\left( \frac{a}{b}\right) ^{\frac{qt}{2}}\left(
\left( 1-t\right) ^{s}\left\vert f^{\prime }\left( b\right) \right\vert
^{q}+t^{s}\left\vert f^{\prime }\left( \sqrt{ab}\right) \right\vert
^{q}\right) dt\right) ^{\frac{1}{q}}
\end{eqnarray*}%
\begin{eqnarray*}
&\leq &\frac{\ln \left( \frac{b}{a}\right) }{4}\left( \frac{q-1}{2q-1}%
\right) ^{1-\frac{1}{q}}\left[ a\left\{ c_{7}(s,q/2)\left\vert f^{\prime
}\left( a\right) \right\vert ^{q}+c_{8}(s,q/2)\left\vert f^{\prime }\left( 
\sqrt{ab}\right) \right\vert ^{q}\right\} ^{\frac{1}{q}}\right.  \\
&&\left. +b\left\{ c_{9}(s,q/2)\left\vert f^{\prime }\left( b\right)
\right\vert ^{q}+c_{10}(s,q/2)\left\vert f^{\prime }\left( \sqrt{ab}\right)
\right\vert ^{q}\right\} ^{\frac{1}{q}}\right] .
\end{eqnarray*}%
b) (1) Since $\left\vert f^{\prime }\right\vert ^{q}$ is GA-$s$-convex
function in the first sense on $\left[ a,b\right] $, from the inequality (%
\ref{3-8a}), we have%
\begin{equation*}
\left\vert \frac{f(a)+f(b)}{2}-\frac{1}{\ln b-\ln a}\dint\limits_{a}^{b}%
\frac{f(x)}{x}dx\right\vert 
\end{equation*}%
\begin{eqnarray*}
&\leq &\frac{a\ln \left( \frac{b}{a}\right) }{2}\left( \frac{q-1}{2q-1}%
\right) ^{1-\frac{1}{q}}\left( \dint\limits_{0}^{1}\left( \frac{b}{a}\right)
^{qt}\left( \left( 1-t^{s}\right) \left\vert f^{\prime }\left( a\right)
\right\vert ^{q}+t^{s}\left\vert f^{\prime }\left( b\right) \right\vert
^{q}\right) dt\right) ^{\frac{1}{q}} \\
&&+\frac{b\ln \left( \frac{b}{a}\right) }{2}\left( \frac{q-1}{2q-1}\right)
^{1-\frac{1}{q}}\left( \dint\limits_{0}^{1}\left( \frac{a}{b}\right)
^{qt}\left( \left( 1-t^{s}\right) \left\vert f^{\prime }\left( b\right)
\right\vert ^{q}+t^{s}\left\vert f^{\prime }\left( a\right) \right\vert
^{q}\right) dt\right) ^{\frac{1}{q}}
\end{eqnarray*}%
\begin{eqnarray*}
&\leq &\frac{\ln \left( \frac{b}{a}\right) }{2}\left( \frac{q-1}{2q-1}%
\right) ^{1-\frac{1}{q}}\left[ a\left\{ c_{11}(s,q)\left\vert f^{\prime
}\left( a\right) \right\vert ^{q}+c_{8}(s,q)\left\vert f^{\prime }\left(
b\right) \right\vert ^{q}\right\} ^{\frac{1}{q}}\right.  \\
&&\left. +b\left\{ c_{12}(s,q)\left\vert f^{\prime }\left( b\right)
\right\vert ^{q}+c_{10}(s,q)\left\vert f^{\prime }\left( a\right)
\right\vert ^{q}\right\} ^{\frac{1}{q}}\right] 
\end{eqnarray*}

(2) Since $\left\vert f^{\prime }\right\vert ^{q}$ is GA-$s$-convex function
in the first sense on $\left[ a,b\right] $, the inequality (\ref{3-6}) is
easily obtained by using the inequality (\ref{3-9a}).
\end{proof}

If taking $s=1$ in Theorem \ref{2.2}, we can derive the following
inequalities for GA-convex.

\begin{corollary}
\label{2.2a}Let $f:I\subseteq \left( 0,\infty \right) \mathbb{\rightarrow R}$
be differentiable on $I^{\circ }$, and $a,b\in I^{\circ }$ with $a<b$ and $%
f^{\prime }\in L\left[ a,b\right] .$ If $\left\vert f^{\prime }\right\vert
^{q}$ is GA-convex function in the second sense on $\left[ a,b\right] $ for $%
q>1$, then%
\begin{equation}
\left\vert \frac{f(a)+f(b)}{2}-\frac{1}{\ln b-\ln a}\dint\limits_{a}^{b}%
\frac{f(x)}{x}dx\right\vert  \label{3-12a}
\end{equation}%
\begin{eqnarray*}
&\leq &\frac{\ln \left( \frac{b}{a}\right) }{2}\left( \frac{q-1}{2q-1}%
\right) ^{1-\frac{1}{q}}\left[ a\left\{ c_{7}(1,q)\left\vert f^{\prime
}\left( a\right) \right\vert ^{q}+c_{8}(1,q)\left\vert f^{\prime }\left(
b\right) \right\vert ^{q}\right\} ^{\frac{1}{q}}\right. \\
&&\left. +b\left\{ c_{9}(1,q)\left\vert f^{\prime }\left( b\right)
\right\vert ^{q}+c_{10}(1,q)\left\vert f^{\prime }\left( a\right)
\right\vert ^{q}\right\} ^{\frac{1}{q}}\right] ,
\end{eqnarray*}%
\begin{equation}
\left\vert f\left( \sqrt{ab}\right) -\frac{1}{\ln b-\ln a}%
\dint\limits_{a}^{b}\frac{f(x)}{x}dx\right\vert  \label{3-12b}
\end{equation}%
\begin{eqnarray*}
&\leq &\frac{\ln \left( \frac{b}{a}\right) }{4}\left( \frac{q-1}{2q-1}%
\right) ^{1-\frac{1}{q}}\left[ a\left\{ c_{7}(1,q/2)\left\vert f^{\prime
}\left( a\right) \right\vert ^{q}+c_{8}(1,q/2)\left\vert f^{\prime }\left( 
\sqrt{ab}\right) \right\vert ^{q}\right\} ^{\frac{1}{q}}\right. \\
&&\left. +b\left\{ c_{9}(1,q/2)\left\vert f^{\prime }\left( b\right)
\right\vert ^{q}+c_{10}(1,q/2)\left\vert f^{\prime }\left( \sqrt{ab}\right)
\right\vert ^{q}\right\} ^{\frac{1}{q}}\right] ,
\end{eqnarray*}%
where $c_{7}$, $c_{8}$, $c_{9}$, $c_{10}$ are defined by (\ref{3-91}) and (%
\ref{3-101}).
\end{corollary}

\section{Application to Special Means}

\begin{proposition}
Let $0<a<b,$and $q\geq 1.$ Then 
\begin{eqnarray*}
&&\left\vert A\left( a,b\right) -L\left( a,b\right) \right\vert \leq \left[
\ln \left( \frac{b}{a}\right) \right] ^{1-\frac{1}{q}}\left( \frac{1}{2}%
\right) ^{2-\frac{1}{q}}\left( \frac{1}{q}\right) ^{\frac{1}{q}} \\
&&\times \left[ \left\{ b^{q}-L(a^{q},b^{q})\right\} ^{\frac{1}{q}}+\left\{
L(a^{q},b^{q})-a^{q}\right\} ^{\frac{1}{q}}\right]
\end{eqnarray*}%
\begin{eqnarray*}
&&\left\vert G\left( a,b\right) -L\left( a,b\right) \right\vert \leq \left[
\ln \left( \frac{b}{a}\right) \right] ^{1-\frac{1}{q}}\left( \frac{1}{2}%
\right) ^{3-\frac{1}{q}}\left( \frac{2}{q}\right) ^{\frac{1}{q}} \\
&&\times \left[ \sqrt{a}\left\{ b^{q/2}-L(a^{q/2},b^{q/2})\right\} ^{\frac{1%
}{q}}+\sqrt{b}\left\{ L(a^{q/2},b^{q/2})-a^{q/2}\right\} ^{\frac{1}{q}}%
\right] .
\end{eqnarray*}
\end{proposition}

\begin{proof}
The assertion follows from the inequalities (\ref{3-7a}) and (\ref{3-7b}) in
Corollary \ref{2.1a} for $f(x)=x,\ x>0$.
\end{proof}

\begin{proposition}
Let $0<a<b\leq 1,$and $q>1.$ Then 
\begin{equation*}
\left\vert A\left( a,b\right) -L\left( a,b\right) \right\vert \leq \ln
\left( \frac{b}{a}\right) \left( \frac{q-1}{2q-1}\right) ^{1-\frac{1}{q}}L^{%
\frac{1}{q}}(a^{q},b^{q})
\end{equation*}%
\begin{equation*}
\left\vert G\left( a,b\right) -L\left( a,b\right) \right\vert \leq \frac{1}{2%
}\ln \left( \frac{b}{a}\right) \left( \frac{q-1}{2q-1}\right) ^{1-\frac{1}{q}%
}L^{\frac{1}{q}}(a^{q},b^{q})A\left( \sqrt{a},\sqrt{b}\right) .
\end{equation*}
\end{proposition}

\begin{proof}
The assertion follows from the inequalities (\ref{3-12a}) and (\ref{3-12b})
in Corollary \ref{2.2a} for $f(x)=x,\ x>0$.
\end{proof}

\end{document}